\documentclass[12pt]{article}
\usepackage{a4}
\usepackage{amsthm}
\usepackage{amsfonts}
\usepackage{amssymb}
\usepackage{amsmath} 
\usepackage{array}
\usepackage[utf8]{inputenc}
\usepackage{graphicx}
\usepackage{cite}
\usepackage{hyperref}
\usepackage{enumitem}
\usepackage{comment} 
\usepackage{algpseudocode}
\usepackage{algorithm}
\usepackage{setspace}
\usepackage{color}
\usepackage{paths}

\title{Edge-decomposing graphs into coprime forests~\thanks{Both authors were partially supported by ANR project Stint under reference ANR-13-BS02-0007 and by the LABEX MILYON (ANR-10-LABX-0070) of Université de Lyon, within the program Investissements d'Avenir (ANR-11-IDEX-0007) operated by the French National Research Agency (ANR). Klimošová was also supported by Center of Excellence – ITI, project P202/12/G061 of GA ČR and by Center for Foundations of Modern Computer Science (Charles Univ. project UNCE/SCI/004).}~\thanks{Extended abstract of this work was published as: T. Klimo\v{s}ov\'a, S. Thomass\'e: Decomposing graphs into paths and trees, Electron. Notes Discrete Math., 61 (2017) 751-757.}}
\author{Tereza Klimo\v{s}ov\'a\thanks{Department of Applied Mathematics, Faculty of Mathematics and Physics, Charles University,  
Malostransk\'e n\'am\v{e}st\'i 25, 118 00 Praha 1, Czech Republic. E-mail: {\tt tereza@kam.mff.cuni.cz}.} \and St\'ephan Thomass\'e\thanks{Laboratoire d’Informatique du Parall\'elisme,
\'Ecole Normale Sup\'erieure de Lyon,
69364 Lyon Cedex 07, France. E-mail: {\tt stephan.thomasse@ens-lyon.fr}.}~\thanks{Institut Universitaire de France}}
\begin{document}
\maketitle

\begin{abstract}

The Bar\'at-Thomassen conjecture, recently proved in~\cite{bib:BHLMT}, asserts that for every tree $T$, there is a constant $c_T$ such that every $c_T$-edge connected graph $G$ with number of edges (size) divisible by the size of $T$ admits an edge partition into copies of $T$ (a $T$-decomposition). In this paper, we investigate in which case the connectivity requirement can be dropped to a minimum degree condition. For instance,
it was shown in~\cite{bib:orig-paths} that when $T$ is a path with $k$ edges, there is a constant $d_k$ such that every $24$-edge connected graph $G$ with size divisible by $k$ and minimum degree $d_k$ has a $T$-decomposition. We show in this paper that when $F$ is a coprime forest (the sizes of its components being a coprime set of integers), any graph $G$ with sufficiently large minimum degree has an $F$-decomposition provided that the size of $F$ divides the size of $G$ (no connectivity is required). A natural conjecture asked in~\cite{bib:orig-paths} asserts that for a fixed tree $T$, any graph $G$ of size divisible by the size of $T$ with sufficiently high minimum degree has a $T$-decomposition, provided that $G$ is sufficiently highly connected in terms of the maximal degree of $T$. The case of maximum degree 2 is answered by paths. We provide a counterexample to this conjecture in the case of maximum degree 3.

%In~\cite{bib:orig-paths}, the authors posed a conjecture that for a fixed tree $T$, every graph with sufficiently high degree and with number of edges divisible by $|E(T)|$ has a $T$-decomposition if it is sufficiently highly edge-connected in terms of maximal degree of $T$. They proved the conjecture for trees of maximum degree $2$ (that is, paths). We disprove the conjecture for trees of maximal degree at least $3$.
%We also prove a relaxed version of the conjecture, showing that for two fixed trees $T_1$ and $T_2$ with coprime numbers of edges, every connected graph with sufficiently high degree has a $T_1,T_2$-decomposition.
\end{abstract}

\section{Introduction}\label{sec:intro}
Given a graph $G$, we denote by $V(G)$ and by $E(G)$ its vertex set and its edge set, respectively. For $X\subseteq V(G)$, $G[X]$ denotes the induced subgraph of $G$ on $X$. Unless we specify otherwise, we consider graphs to be simple graphs without loops and multigraphs to have multiple edges and loops.

For graphs $G$ and $H$, we say that $G$ is {\em $H$-decomposable} if there exists a partition $\{E_i\}_{i\in[k]}$ of $E(G)$ such that every $E_i$ forms an isomorphic copy of $H$. We then call $\{E_i\}_{i\in[k]}$ an {\em $H$-decomposition} of $G$. Note that if $G$ has an $H$-decomposition, $|E(H)|$ divides $|E(G)|$.
More generally, if there exists a partition $\{E_i\}_{i\in[k]}$ of $E(G)$ such that each $E_i$ forms an isomorphic copy of one of the graphs $H_1,\ldots, H_j$, we call $\{E_i\}_{i\in[k]}$ a $H_1,\ldots, H_j$-decomposition of $G$.

In~\cite{bib:barat-thomassen}, Bar\'at and Thomassen conjectured that for a fixed tree $T$, every sufficiently edge-connected graph with number of edges divisible by $|E(T)|$ has a $T$-decomposition.

\begin{conjecture}\label{conj:barat-thomassen}
 For any tree $T$ on $m$ edges, there exists an integer $k_T$ such that every $k_T$-edge-connected graph with number of edges divisible by $m$ has a $T$-decomposition.
\end{conjecture}

They also observed a relation between $T$-decompositions and Tutte's conjecture, which states that every $4$-edge-connected graph admits a nowhere-zero $3$-flow. Until recently it was not even known that any constant edge-connectivity would suffice.
Bar\'at and Thomassen have shown that if every $8$-edge-connected graph has a $K_{1,3}$-decomposition, then every $8$-edge-connected graph has a nowhere-zero $3$-flow and, vice versa, Tutte’s $3$-flow conjecture would imply that every $10$-edge-connected graph with number of edges divisible by 3 has a $K_{1,3}$-decomposition.

%Since then, Conjecture 1 has attracted growing attention, and it has
%now been verified for different families of trees such as stars [14], some bistars [1, 16], and
%paths of a certain length [6, 12, 13, 15]. Very recently, breakthrough results were obtained
%by Merker [10], who proved the conjecture for all trees of diameter at most 4, hence covering
%some of the results above, and by Botler, Mota, Oshiro, and Wakabayashi [5], who proved
%the conjecture for all paths. The latter result was improved by Bensmail, Harutyunyan, Le,
%and Thomass\'e [4], who showed that, for path-decompositions, large minimum degree is a
%sufficient condition provided the graph is 24-edge-connected.

A series of results showing that Conjecture~\ref{conj:barat-thomassen} holds for specific trees followed~\cite{bib:BT-4, bib:thomassen-paths, bib:BT-bistars, bib:BT-diam4, bib:BT-paths3, bib:BT-paths2, bib:BT-paths5, bib:BT-botler-all-paths, bib:thomassen-3-flow, bib:orig-paths}. 
 %been verified for different families of trees such as stars [14], some bistars [1, 16], and
%paths of a certain length [6, 12, 13, 15]. Very recently, breakthrough results were obtained
%by Merker [10], who proved the conjecture for all trees of diameter at most 4, hence covering
%some of the results above, and by Botler, Mota, Oshiro, and Wakabayashi [5], who proved
%the conjecture for all paths. The latter result was improved by Bensmail, Harutyunyan, Le,
%and Thomass\'e [4], who showed that, for path-decompositions, large minimum degree is a
%sufficient condition provided the graph is 24-edge-connected.%TODO biblio
%TODO stars and 3-flow conj 
Recently, the conjecture was proven by Bensmail, Harutyunyan, Le, Merker and the second author~\cite{bib:BHLMT}.

\begin{theorem}\label{thm:barat-thomassen}
For any tree $T$, there exists an integer $k_T$ such that every $k_T$-edge-connected graph with number of edges divisible by $|E(T)|$ has a $T$-decomposition.
\end{theorem}

In~\cite{bib:orig-paths}, the authors posed the following, strengthened version of the conjecture of Bar\'at and Thomassen and they proved it for $T$ being a path.

\begin{conjecture}\label{conj:wrong}
There is a function $f$ such that, for any fixed tree $T$ with
maximum degree $\Delta_T$, every $f(\Delta_T)$-edge-connected graph with  minimum degree at least $f(|E(T)|)$ and number of edges divisible by $|E(T)|$ has a $T$-decomposition.
\end{conjecture}

We define the {\em length} of a path as its number of edges. The following result from~\cite{bib:orig-paths} answers the previous question for $\Delta_T=2$.

\begin{theorem}\label{thm:orig-paths}
For every integer $\ell$, there exists $d=d(\ell)$ such that the edge set of every $24$-edge-connected graph $G$ with minimum degree $d$ and number of edges divisible by $\ell$ has a decomposition into paths of length $\ell$.
\end{theorem}

However, Conjecture~\ref{conj:wrong} is not true in general. In Section~\ref{sec:disprove}, we show that it does not hold even for trees of maximum degree three. 

In Section~\ref{sec:twotrees}, we consider a variation of Conjecture~\ref{conj:wrong} for forests. We call a tree {\em proper} if it has at least one edge. We call a forest {\em proper} if each of its connected components is a proper tree. Note that if $F$ is a forest which is not proper, for the purpose of finding an $F$-decomposition, one can disregard components with a single vertex, provided that the host graph has at least as many vertices as $F$. Having this in mind, we state our result only for proper forests.

Theorem~\ref{thm:barat-thomassen} can be easily extended to proper forests by gluing several copies of a forest together (see~\ref{sec:twotrees} for details). Moreover, one can replace edge-connectivity requirement by minimum degree requirement in case of {\em coprime forests}.
%Extending the method described above, we show that for proper coprime forests, one can replace edge-connectivity requirement by minimum degree requirement. 
We call a forest {\em coprime} if there is no integer $d>1$ which divides the number of edges of all of its components.

\begin{theorem}\label{thm:forest}
For any proper coprime forest $F$, there exists an integer $\delta=\delta(F)$ such that every graph with minimum degree at least $\delta$ and number of edges divisible by the number of edges of $F$ has an $F$-decomposition.
\end{theorem}

Note that the requirement that the forest is coprime is necessary. For instance, if every component of $F$ has even number of edges, a graph with a connected component that has an odd number of edges does not have any $F$-decomposition.

From Theorem~\ref{thm:forest}, we can easily derive, in a sense, a relaxation of Conjecture~\ref{conj:wrong}. In particular, for two trees $T_1$, $T_2$ with coprime numbers of edges, high minimum degree implies the existence of a $T_1,T_2$-decomposition.

\begin{corollary}\label{thm:twotrees}
For any two trees $T_1, T_2$ with coprime numbers of edges, there exists an integer $\delta=\delta(T_1,T_2)$ such that every graph with minimum degree at least $\delta$ has a $T_1,T_2$-decomposition.
\end{corollary}

The core idea of the proof of Theorem~\ref{thm:forest} (which is used in Lemma~\ref{lem:3tree}) is to partition the host graph along edge-cuts of bounded size into vertex disjoint parts that are highly edge-connected. This is achieved by repeated applications of the following lemma.

\begin{lemma}\label{lem:cut}
Let $k$ be an integer and $G$ be a multigraph. Then, there is a cut $(A,B)$ in $G$ of order at most $2k$ such that $G[A]$ is $k$-edge-connected or has only one vertex.
\end{lemma}

We then obtain decompositions of these parts using the result of~\cite{bib:BHLMT}. Finally, we construct a decomposition of the whole graph by combining decompositions of the parts.

%TODO Tutte 3-flow

\section{Disproving Conjecture~\ref{conj:wrong}}\label{sec:disprove}
%\section{Disproving Conjecture~\ref{conj:wrong}}\label{sec:disprove}

In this section, we disprove Conjecture~\ref{conj:wrong}. We show that it does not hold even for trees of maximum degree three.%, in particular, that a complete binary tree is a counterexample. 

\begin{figure}
\begin{center}
\includegraphics[scale=1]{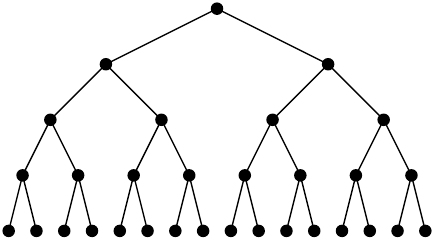}
\caption{The tree $T_4$.}
\label{fig:tree}
\end{center}
\end{figure}

Assume that Conjecture~\ref{conj:wrong} holds for some function $f$. Let $T_k$ be the complete binary tree of depth $k$ (see Figure~\ref{fig:tree} for an example). The maximum degree of $T_k$ is three and the number of edges of $T_k$ is $n_k=2^{k+1}-2$ for every $k$. Let $\bbT_k$ be the set of possible numbers of edges contained in a component of $T_k\setminus e$ for some edge $e$ of $T_k$.
Observe that the components of $T_k\setminus e$ have $2^i-2$ edges and $n_k-(2^i-1)=2^{k+1}-2^i-1$ edges for some $i\in[k]$, for every edge $e$. Thus, $\bbT_k=\{2^i-2|i\in[k]\}\cup \{2^{k+1}-2^i-1|i\in[k]\}$ and then $|\bbT_k|\leq 2k$.

It follows that the sum $\sum_{i=1}^{f(3)} t_i$, where $t_i\in \bbT_k$ for every $i\in [f(3)]$, can attain at most $(2k)^{f(3)}$ different values. Therefore, there exists $k_0$ such that $(2k)^{f(3)}<n_k$ for every $k>k_0$. 

Fix $k>k_0$. Then, there exists $m\in [n_k]$ such that $m\neq  \sum_{i=1}^{f(3)} t_i \mod n_k$ for any choice of values of $t_i\in \bbT_k$.

Let $G_1$ and $G_2$ be $f(3)$-edge-connected graphs with minimum degree at least $f(|E(T_k)|)$ such that the number of edges of $G_1$ is congruent to $m$ modulo $n_k$, the number of edges of $G_2$ is congruent to $n_k-f(3)-m$ modulo $n_k$ and there are sets $S_1\subseteq V(G_1)$, $S_2\subseteq V(G_2)$ of size $f(3)$ such that the distance between every two distinct vertices of $S_i$ in $G_i$, $i=1,2$, is greater than $2k$ (where the distance is the length of the shortest path between the two vertices). See Figure~\ref{fig:star} for an example of a construction of $G_1$ with these properties, $G_2$ can be constructed in an analogous way. 

\begin{figure}
\begin{center}
\includegraphics[scale=1]{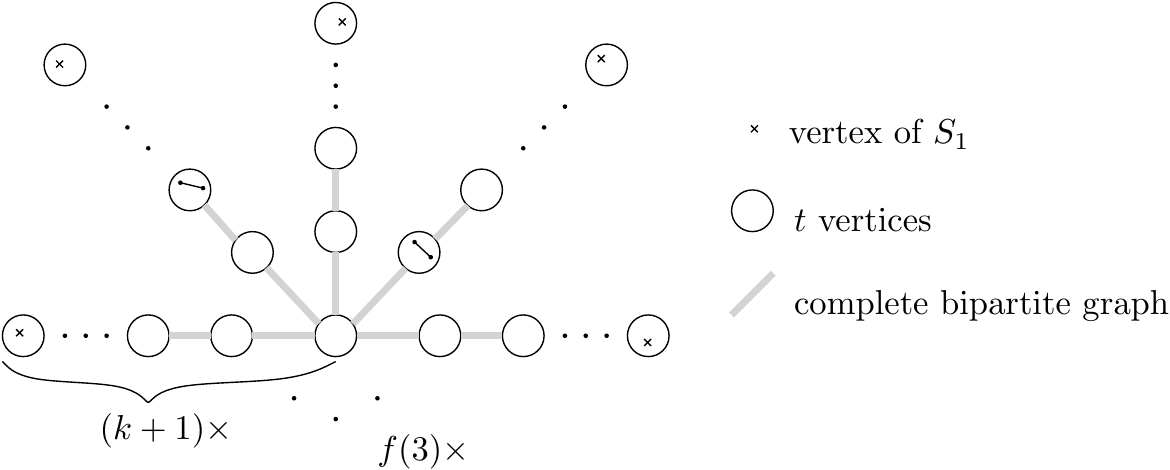}
\caption{A graph $G_1$ with the desired properties can be easily constructed by taking a star with $f(3)$ edges, subdividing each edge $k$ times, replacing each vertex by $t=\max(f(|E(T_k)|),n_k)$ vertices and each edge by a complete bipartite graph and finally, adding appropriate number of additional edges between the vertices corresponding to the same vertex of the subdivided star, so that the number of edges of the graph is congruent to $m$ modulo $n_k$. Taking vertices of $S_1$ as in the figure yields a set of size $f(3)$ with distance at least $2k+2$ between every two vertices.}
\label{fig:star}
\end{center}
\end{figure}

Let $G$ be a graph obtained from the disjoint union of $G_1$ and $G_2$ by adding a matching $M$ of size $f(3)$ between the vertices of $S_1$ and $S_2$. Then, $G$ is $f(3)$-edge-connected, of minimum degree at least $f(|E(T_k)|)$ and with number of edges divisible by $|E(T_k)|$. Assume that there exists a $T_k$-decomposition $\TT$ of $G$. Since the distance between any two vertices in $S_1$ and any two vertices in $S_2$ is greater than the distance between any two vertices in $T$, every copy of $T_k$ in $\TT$ contains at most one vertex of $S_1$ and $S_2$ and therefore at most one edge of $M$. Note that each copy of $T$ with an edge in $M$ contains $t_i\in \bbT_k$ edges of $G_1$. Therefore, the number of edges of $G_1$ is $cn_k+\sum_{i=1}^{f(3)} t_i$, where $c$ is an integer and $t_i\in \bbT_k$ for every $i\in [f(3)]$. This yields a contradiction with the choice of the number of edges of~$G_1$.

\section{Decomposition into coprime trees}\label{sec:twotrees}
In this section, we prove Theorem~\ref{thm:forest} after introducing some tools.

Let $\GG$ be a set of graphs. We say that $G$ has a $\GG$-decomposition if its edge set can be partitioned into edge disjoint copies of graphs in $\GG$. Decompositions are transitive in the following sense.

\begin{observation}\label{obs:trans}
Let $G$ be a graph and let $\GG$ and $\HH$ be sets of graphs. If the edge set of $G$ can be partitioned in edge-disjoint subgraphs such that each of them has an $\HH$-decomposition, then $G$ has an $\HH$-decomposition. In particular, if every graph in $\GG$ has an $\HH$-decomposition and $G$ has a $\GG$-decomposition, $G$ has an $\HH$-decomposition.
\end{observation}

Given a sequence of vertex disjoint proper trees $T_1,\ldots, T_k$, we define a {\em chain} $T_1\circ T_2\circ \cdots \circ T_k$ as a graph obtained from $T_1,\ldots, T_k$ by choosing two distinct leaves $u_i$ and $v_i$ in each $T_i$ and identifying $u_i$ and $v_{i+1}$ for $i\in [k-1]$. Note that a chain is a tree.

For a proper tree $T$ and an integer $k$, we define a {\em $k$-chain of $T$} to be a chain $T_1\circ T_2 \circ \cdots \circ T_k$, where each $T_i$ is isomorphic to $T$. We denote it by $\chain{T}{k}$. Similarly, we define a {\em $k$-chain} $\chain{F}{k}$ of a proper forest $F$ with connected components $T_1,\ldots, T_m$ as the tree $\chain{T_1}{k}\circ \cdots \circ \chain{T_m}{k}$. 

\begin{observation}\label{obs:chain-dec}
Let  $F$ be a proper forest and $k\geq 2$. Then $\chain{F}{k}$ has an $F$-decomposition.
\end{observation}

%obtain from $k$ copies $T_1,\ldots T_k$ of $T$ by choosing two distinct leaves $u_i$ and $v_i$ in each $T_i$ and identifying $u_i$ and $v_{i+1}$ for each $i\in [k-1]$ (see Figure~\ref{fig:chain} for illustration).(Note that $\circ$ is not associative and in particular, $T\circ T\circ T$ might not be $circ^k T$). 

Next, we argue that Theorem~\ref{thm:barat-thomassen} can be easily extended to proper forests. In particular, given a fixed proper forest $F$, if $G$ has number of edges divisible by $|E(F)|$ and is $k_{\chain{F}{2}}+|E(F)|$-edge-connected, where $k_{\chain{F}{2}}$ is as in Theorem~\ref{thm:barat-thomassen}, then $G$ has an $F$-decomposition. Indeed, from Theorem~\ref{thm:barat-thomassen} it follows that if $|E(G)|$ is divisible by $|E(\chain{F}{2})|=2|E(F)|$, then $G$ has an $\chain{F}{2}$-decomposition. The existence of an $F$-decomposition follows from Observation~\ref{obs:trans}.

%Given a fixed proper forest $F$ with connected components $T_1, \ldots, T_k$, let $H$ be a chain $T'_1\circ T''_1\circ T'_2\circ T''_2\circ \cdots \circ T'_k\circ T''_k$, where $T'_i$ and $T''_i$ are vertex disjoint copies of $T_i$ for each $i\in [k]$. 
If $|E(G)|$ is divisible by $|E(F)|$ but not by $|E(\chain{F}{2})|$, we can remove any  copy of $F$ to make the number of edges of $G$ divisible by $|E(\chain{F}{2})|$, decreasing edge-connectivity by at most $|E(F)|$ (to at least $k_{\chain{F}{2}}$) and then argue as before.

For the proof of Theorem~\ref{thm:forest}, we need the following lemma, which we prove in Section~\ref{sec:3tree}.

\begin{lemma}\label{lem:3tree}
Let $T$, $T_1$ and $T_2$ be trees such that the number of edges of $T$ is coprime with both the number of edges of $T_1$ and the number of edges of $T_2$. Then, there exists an integer $\delta=\delta(T,T_1,T_2)$ such that every graph with minimum degree at least $\delta$ and number of edges divisible by the greatest common divisor of $|E(T)|$ and $|E(T_1)|+|E(T_2)|$ has a $\{T,T_1,T_2\}$-decomposition that contains the same number of copies of $T_1$ and $T_2$.
\end{lemma}

%one can construct a tree $T$ such that its edge sets decomposed into two disjoint copies of $F$ 
%by picking two distinct vertices $u_i$ and $v_i$ in the component $T_i$ of the first copy of $F$ and $u'_i$ and $v'_i$ in the component $T_i$ of the second copy of $F$ and identifying $v_i$ and $u'_i$ for every $i\in [k]$ and $v'_i$ and $u_{i+1}$ for every $i\in [k-1]$
%Then, the existence of a $T$-decomposition implies the existence of an $F$-decomposition.

%We denote the number of edges of a graph $G$ by $\ec{G}$.
\begin{proof}[Proof of Theorem~\ref{thm:forest}]
Let $K_1,\ldots, K_s$ be the connected components of the coprime forest $F$. By B\'ezout's lemma, there exist integers $p_1,\ldots, p_s$ such that $\sum_{i=1}^{s} p_i\ec{K_i}=1$. 
We define $n:=2|V(F)|-\min \{0,p_1,\ldots, p_s\}$ and $m:=2|V(F)|+n+\max \{0,p_1,\ldots, p_s\}$. Let $T_1$ be $K_1^{\circ n+p_1}\circ \cdots \circ K_s^{\circ n+p_s}$ and $T_2$ be $K_1^{\circ m-(n+p_1)}\circ \cdots \circ K_s^{\circ m-(n+p_s)}$. Then, $\ec{T_1}=n\ec{F}+1$, i.e. the number of edges of $T_1$ is coprime with the number of edges of $F$. 
Moreover, $\ec{T_2}=m\ec{F}-\ec{T_1}=(m-n)\ec{F}-1$, thus the number of edges of $T_2$ is also coprime with number of edges of $F$. Let $r$ be a prime number greater than $\ec{T_1}$ and $\ec{T_2}$ and define $T$ to be $F^{\circ r}$. Then, $\ec{T}$ is coprime with both $\ec{T_1}$ and $\ec{T_2}$.

Then, by Lemma~\ref{lem:3tree}, $G$ has a $\{T,T_1,T_2\}$-decomposition $\TT$. Let $G'$ be a subgraph of $G$ formed by the copies of $T_1$ and $T_2$ in $\TT$. Since the graph $G\setminus E(G')$ has a $T$-decomposition by Observation~\ref{obs:trans} and~\ref{obs:chain-dec}, it also has an $F$-decomposition. Thus, to show that $G$ has an $F$-decomposition, it suffices to show that $G'$ has an $F$-decomposition.

Note that $G'$ has a $\{T_1,T_2\}$-decomposition into edge disjoint copies $T_1^1,\ldots, T_1^{\ell}$
 of $T_1$ and edge disjoint copies $T_2^1,\ldots, T_2^{\ell}$ of $T_2$ for some integer 
 $\ell$. Let $H_i=T_1^i\cup T_2^i$. We show that $H_i$ has an $F$-decomposition for every 
 $i\in [\ell]$ and thus, by Observation~\ref{obs:trans}, $G'$ has an $F$-decomposition. By construction,  $T_1^i$ has a $(K_1,\ldots, K_s)$-decomposition $\TT_1$ in which each $K_j$ appears $n+p_j$-times. 
 Similarly, $T_2^i$ has a $(K_1,\ldots, K_s)$-decomposition $\TT_2$ in which each $K_j$ appears $m-(n+p_j)$-times. 
 Note that  $n+p_j$ and $m-(n+p_j)$ are at least $2|V(F)|$ for every $j\in[s]$. We first construct $m-4|V(F)|$ copies of $F$ in $H_i$ combining trees from $\TT_1$ and $\TT_2$ in such a way that at the end of the process, 
 there will be exactly $2|V(F)|$ unused copies of each $K_i$ in each of $\TT_1$ and $\TT_2$.
  We  construct one copy of $F$ at a time by greedily picking copies of $K_1, \ldots, K_s$ one by one, always choosing it from that of $\TT_1$ and $\TT_2$ which contains the greatest number of unused copies of the component (in case of equality we choose arbitrarily).
   We argue that it is always possible to pick a copy which is vertex disjoint from all the previously picked components in the currently constructed copy of $F$. Assume that we need to pick a copy of $K_j$ and assume that $\TT_1$ contains greater number of unused copies of $K_j$ - the argument for $\TT_2$ is analogous.

The total number of copies of $K_j$ in $\TT_{1}$ and $\TT_{2}$ is $m$ and therefore after using less than $m-4|V(F)|$ copies of $K_j$, $\TT_1$ contains more than $2|V(F)|$ unused copies of $K_j$. By construction, each copy of $K_k$, $k\neq j$ in $\TT_{1}$ or $\TT_{2}$ intersects at most $2|V(K_k)|$ copies of $K_j$ in $\TT_1$. Therefore, the already chosen components of the currently constructed copy of $F$ intersect at most $2|V(F)|$ copies of $K_j$ in $\TT_1$, so there exist at least one copy of $K_j$ which is vertex disjoint. Note that after constructing $m-4|V(F)|$ copies of $F$ in this way, both $\TT_1$ and $\TT_2$ contain $2|V(F)|$ unused copies of each $K_j$, $j\in[s]$. We denote these collections of unused copies by $\TT'_1$ and $\TT'_2$ respectively and argue that $\TT'_1$ can be partitioned into copies of $F$. Then, $\TT'_2$ can be partitioned into copies of $F$ by analogous arguments.

We call a copy of $K_i$, $i\in [s]$, in $\TT'_1$ {\em bad} if it intersects a copy of $K_j$ in $\TT'_1$ with $j\neq i$. We call all other copies in $\TT'_1$ {\em good}.
Observe that by construction, there are at most $2s$ bad copies in $\TT'_1$.  Any collection of copies in $\TT'_1$ containing one copy of $K_i$ for each $i\in [s]$ such that at most one of the copies is bad forms a copy of the forest $F$ (i.e., the copies in the collection are pairwise disjoint). Thus, it is possible to construct copies of $F$ greedily by picking one bad copy and good copies of the remaining components of $F$. Since $\TT'_1$ contains $2|V(F)|$ copies of each $K_i$, $i\in [s]$  and $2s<2|V(F)|$, there are enough good copies of each component of $F$ for repeatedly creating copies of $F$ in this way, until all the bad copies in $\TT'_1$ are used. Once all the bad copies are used, the remaining good copies in $\TT'_1$ can be arbitrarily partitioned into copies of $F$.

%Observe that since $s\geq |V(F)|$, there exist orders $K_i^1,\ldots, K_i^{2|V(F)|}$ of copies of $K_i$ in $\TT'_1$ for $i\in [s]$, such that for every $j\in [2|V(F)|]$, the set $K_1^j,\ldots, K_s^j$ contains at most one bad tree. Thus, $K_1^j,\ldots, K_s^j$ are vertex disjoint and therefore form a copy of $F$ for every $j\in  [2|V(F)|]$. This yields a desired partitioning of $\TT'_1$ into copies of $F$. 

%at most $2s$ trees in $\TT'_1$ intersect a tree in $\TT'_1$ corresponding to different component of $F$. Precisely, there are at most $s$ pairs of trees $J$, $J'$ in $\TT'_1$ such that $J$ is a copy of $K_j$, $J'$ is a copy of $K_j$ and $j\neq k$. We call these trees {\em bad}.
\end{proof}

\section{Proof of Lemma~\ref{lem:3tree}}\label{sec:3tree}

We start by showing that sufficiently highly edge-connected graph can be decomposed into copies of two trees with coprime numbers of edges.

\begin{lemma}\label{lem:conn-twotrees}
Let $T_1, T_2$ be trees with coprime numbers of edges. Then, there exists an integer $K=K(T_1,T_2)$ such that every $K$-edge-connected graph has a $\{T_1,T_2\}$-decomposition with less than $|E(T_1)|$ copies of $T_2$.
\end{lemma}

\begin{proof}
Let $m_1$, $m_2$ be numbers of edges of $T_1$, $T_2$, respectively, and let $k_{T_1}$ be as in Theorem~\ref{thm:barat-thomassen} and let $K=k_{T_1}+m_1m_2$. Let $G$ be a $K$-edge-connected graph and let $n$ be the smallest non-negative integer such that $m_1|(|E(G)|-nm_2)$. Since $m_1$ and $m_2$ are coprime, such $n$ exists and is smaller than $m_1$ by B\'{e}zout's Lemma.

By the greedy algorithm, it is possible to find a collection $\TT$ of $n$ edge-disjoint copies of $T_2$ in $G$. Then, $G\setminus E(\TT)$ is a $k_{T_1}$-edge-connected graph with number of edges divisible by $m_1$. The result follows from Theorem~\ref{thm:barat-thomassen}.

\end{proof}

For the purpose of the proof of Lemma~\ref{lem:3tree}, we extend the definition of a graph by allowing hyperedges of size one, which we call {\em stubs} and we call the resulting object a {\em stub graph}. Each vertex of a stub graph can be incident with arbitrarily many stubs and the {\em degree} of a vertex is the number of edges and stubs incident with it. Moreover, we assign a positive integer $i_s$ to each stub $s$ and we call $i_s$ the {\em index of the stub} $s$. \new{Intuitively, a stub can be viewed as a remainder of an edge after  one of its endvertices has been removed from the graph. The index of the stub then contains some information about the removed endvertex.}

We write $E(G)$ to denote the set consisting of the edges and the stubs of a stub graph $G$ and we call it the {\em edge set} of $G$, but we do not refer to stubs as edges otherwise. 
We denote the set of stubs of index $i$ in a stub graph $G$ by $S_i(G)$. A {\em subgraph} of a stub graph is also a stub graph and we say that two subgraphs are edge-disjoint if their edge sets are disjoint, i.e., they do not share any edge or stub.

Let $\TT$ be a set of trees (without stubs). We extend the definition of $\TT$-decomposition for graphs to stub graphs. Informally, in a $\TT$-decomposition of a stub graph, a stub incident with a vertex $v$ plays the role of a subtree of $T\in \TT$, such that the vertex $v$ is a leaf of this subtree.

A {\em twig} is a pair $(T,r)$, where $T$ is a proper tree and $r$ is a leaf of $T$. We say that $r$ is the {\em root} of the twig.
Let $G$ be a stub graph, $s$ a stub incident with a vertex $v$ in $G$ and $(T,r)$ a twig disjoint from $G$. A stub graph obtained from $G$ by {\em expanding $s$ by $(T,r)$} is the stub graph obtained from $G\setminus s$ and $T$ by identifying $v$ and $r$. 

An {\em embedding} of a tree $T$ in a stub graph $G$ is a subgraph $T'$ of $G$ such that each stub in $T'$ has different index and there exists a mapping $\st$ assigning a twig $(T_s,r_s)$ to each stub $s$ in $T'$ such that expanding every stub $s$ in $T'$ by $\st(s)$ yields a copy of $T$. See Figure~\ref{fig:stub-dec} for an example. 

We say that a stub graph $G$ has a $\TT$-decomposition if its edge set can be decomposed into disjoint sets $\{E_i\}_{i\in [k]}$ such that each $E_i$ forms an embedding of some $T\in \TT$. Note that this definition coincides with the definition of a $\TT$-decomposition in the usual sense if $G$ has no stubs. 

We denote the graph obtained from $G$ by removing all the stubs by $G^{-}$ and we say that $G$ is {\em $k$-edge-connected} if $G^{-}$ is $k$-edge-connected. 
The next observation asserts that a $\TT$-decomposition of $G^{-}$ can be easily extended to a $\TT$-decomposition of $G$.

\begin{observation}\label{obs:stub-dec}
If $G^{-}$ has a $\TT$-decomposition $\TT_1$, there exists a $\TT$-decomposition $\TT_2$ of $G$ such that $\TT_1\subseteq \TT_2$. Moreover, given $T\in \TT$, there exists $\TT_2$ such that $\TT_2\setminus \TT_1$ contains only embeddings of $T$.
\end{observation}

\begin{proof}
It follows from the fact that a stub $s$ (with the vertex incident to it) forms an embedding of any proper tree $T$. It is enough to let $\st(s)=(T,r)$, where $r$ is some leaf of $T$.
\end{proof}

In particular, a stub graph $G$ with no edges has a $\TT$-decomposition for any set of proper trees $\TT$, since $G^{-}$ has a trivial (empty) $\TT$-decomposition.
It follows that Lemma~\ref{lem:conn-twotrees} holds for stub graphs as well.

Next, we introduce some more tools and terminology. Let $G$ be a multigraph. We call a partition $(A,B)$ of $V(G)$ into two parts a {\em cut} in $G$. We denote $E(A,B)$ the set of edges of $G$ incident with a vertex in both $A$ and $B$ and call $|E(A,B)|$ the {\em order of the cut} $(A,B)$. We now prove Lemma~\ref{lem:cut}. Note that the definition of a cut and the statement of the lemma trivially extend to stub graphs (by considering the multigraph $G^-$ instead of the stub graph $G$). 

\begin{proof}[Proof of Lemma~\ref{lem:cut}]
Let $(A,B)$ be a cut in $G$ of order at most $2k$ such that $A$ is inclusion-wise minimal. Assume that $A$ has more than one vertex and let $(A_1,A_2)$ be a cut in $G[A]$. By minimality of $A$, we have that the cuts $(A_1,V(G)\setminus A_1)$ and $(A_2,V(G)\setminus A_2)$ have order greater than $2k$. Since $|E(A_1,V(G)\setminus A_1)|+ |E(A_2,V(G)\setminus A_2)|=|E(A,B)|+2|E(A_1,A_2)|$, $(A_1,A_2)$ has order greater than $k$.
\end{proof}

The following results of Czumaj and Strothmann were originally proven only for simple graphs, however, they easily extend to multigraphs.% (it is enough to subdivide multiple edges and loops by auxiliary vertices and apply the original theorem on the resulting graph). %TODO Stephanova poznamka outbranching

\begin{theorem}[Czumaj, Strothmann~\cite{bib:czumaj}, extended]\label{thm:two-sparse-tree}

Every $2$-edge-connected multigraph $G$ contains a spanning tree $T$
such that
$\deg_T v\leq (\deg_G v+3)/2$ for every vertex $v$ of $G$.
\end{theorem}

\new{To find such a tree, it is enough to take an out-branching in a strongly connected balanced orientation of $G$.}

\begin{theorem}[Czumaj, Strothmann~\cite{bib:czumaj}, extended]\label{thm:sparse-tree}

Let $p$ be a positive integer. If a multigraph $G$ contains $2^p$ edge-disjoint spanning trees, then $G$ has a spanning tree $T$ such that
$\deg_T v\leq \deg_G v/2^p+3p/2$ for every vertex $v$ of $G$.
\end{theorem}

The following Corollary~\ref{cor:sparse-tree} of Theorem~\ref{thm:sparse-tree} was essentially proven in~\cite{bib:thomassen-paths}. Our version differs in some details. Among other things, we use the following theorem of Nash-Williams and Tutte to replace the requirement of having a collection of edge-disjoint spanning trees by edge-connectivity. 

\begin{theorem}[Nash-Williams~\cite{bib:nash-williams}, Tutte~\cite{bib:tutte}]\label{thm:nash-williams}
If a multigraph $G$ is $2k$-edge-connected, then $G$ contains $k$ edge-disjoint spanning trees.
\end{theorem}

\begin{corollary}\label{cor:sparse-tree}
For every $\varepsilon>0$ and integer $m$, there exists $L$ such that every $2^n$-edge-connected stub graph $G$ with minimum degree at least $L$, where $n=2+m+ \lceil\log (1/\varepsilon)\rceil$,   
has $2^{m}$ edge-disjoint spanning trees $T_1,\ldots, T_{2^m}$ such that
\[\sum_{i\in [2^m]}\deg_{T_i} v \leq \varepsilon \deg_{G} v\]
 for every $v\in V(G)$.
\end{corollary}

\begin{proof}
By Theorem~\ref{thm:nash-williams}, the graph $G^{-}$ contains $2^{n-1}=2^m\cdot 2^{\lceil \log(1/\varepsilon)\rceil+1}$ edge-disjoint spanning trees. Thus, $G^{-}$ contains $2^m$ edge-disjoint $2^{\lceil \log(1/\varepsilon)\rceil+1}$-edge-connected spanning subgraphs (formed by unions of the spanning trees). From Theorem~\ref{thm:sparse-tree} applied to these $2^{\lceil \log(1/\varepsilon)\rceil+1}$-edge-connected graphs, it follows that $G^{-}$ contains $2^{m}$ edge-disjoint spanning trees $T_1,\ldots,T_{2^m}$ such that \[\sum_{i\in [2^m]}\deg_{T_i}v\leq \frac{\deg_{G^-} v}{2^{\lceil \log(1/\varepsilon)\rceil+1}}+3\cdot 2^{m}(\lceil \log(1/\varepsilon)\rceil+1)\leq \frac{\varepsilon \deg_{G} v}{2}+3\cdot 2^{m}(\lceil \log(1/\varepsilon)\rceil+1)\] for every $v\in V(G)$. 
Moreover $3\cdot 2^m (\lceil \log(1/\varepsilon)\rceil+1) \leq \varepsilon L/2$ for $L$ sufficiently large. The result follows.
\end{proof}

Corollary~\ref{cor:sparse-tree} implies the following.

\begin{corollary}\label{cor:split}
For any positive integers $k_0$ and $\delta_0$, there exists $d_0$ such that the edge set of every $16k_0$-edge-connected stub graph $G$ with minimum degree at least $d_0$ can be decomposed into a $k_0$-edge-connected graph and a stub graph of minimum degree at least $\delta_0$.
\end{corollary}

\begin{proof}
Let $d_0=\max(2\delta_0,L)$, where $L$ is as in Corollary~\ref{cor:sparse-tree} for $\varepsilon=1/2$ and $m=\lceil \log k_0\rceil$. Observe that $2^n$ in Corollary~\ref{cor:sparse-tree} is then less than $16k_0$ (and equal to $8k_0$ if $k_0$ is a power of two). Thus, by Corollary~\ref{cor:sparse-tree}, there exist $k_0$ edge-disjoint spanning trees $T_1,\ldots, T_{k_0}$ such that $\sum_{i\in [k_0]}\deg_{T_i} v \leq 1/2 \deg_{G} v$ for every $v$. By the choice of $d_0$, $1/2 \deg_{G} v\geq \delta_0$ for every $v$ of $G$. Thus, $G\setminus (\bigcup_{i\in[k_0]} E(T_i))$ has minimum degree at least $\delta_0$.
\end{proof}

\begin{figure}
\begin{center}
\includegraphics[scale=1.2]{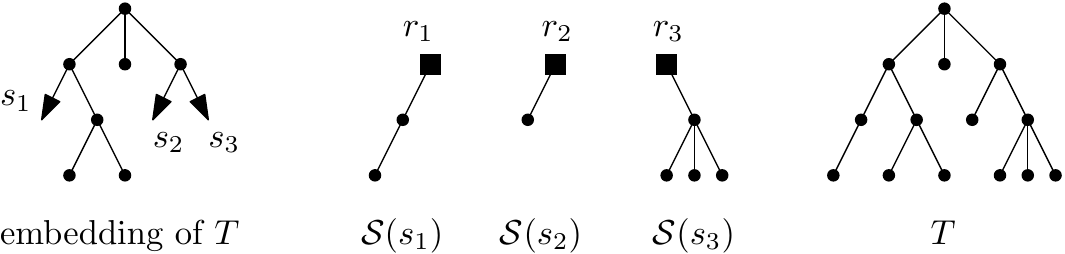}
\caption{Embedding of a tree. Stubs are depicted as arrows, roots are depicted as squares.}
\label{fig:stub-dec}
\end{center}
\end{figure}

We need the following easy consequence of B\'ezout identity for the proof of Lemma~\ref{lem:3tree}.
\begin{observation}\label{obs:bezout}
Let $a,b$ be positive integers and let $c$ be an integer such that $c>ab$ and $c$ is divisible by the greatest common divisor of $a$ and $b$. Then there exist non-negative integers $k_a$ and $k_b$ such that $k_b<a$ satisfying $k_a a +k_b b=c$.
\end{observation}

\begin{proof}
Let $d$ be the greatest common divisor of $a$ and $b$ and let $k_c$ be an integer such that $k_c d=c$. By B\'ezout identity there exist integers $x$ and $y$ satisfying $xa+yb=d$, thus $xk_c a +yk_c b=c$. Note that then also $(xk_c+ib)a +(yk_c-ia) b=c$ for any integer $i$. Let $i$ be such that $0\leq xk_c-ia <a$. Then $k_b=xk_c-ia$ and $k_a=xk_c+ib$ satisfy the claim of the observation, in particular, since $c>ab$ and $k_bb<ab$, $k_a$ must be positive.
\end{proof}

We now prove Lemma~\ref{lem:3tree}.
\begin{proof}[Proof of Lemma~\ref{lem:3tree}]
Let $k_0=\max(K(T,T_1),K(T,T_2), k_T)$, where $K$ is as in Lemma~\ref{lem:conn-twotrees} and $k_T$ is as in Theorem~\ref{thm:barat-thomassen}. Let $m=\max(|E(T)|,|E(T_1)|, |E(T_2)|)$, $d_0$ be as in Corollary~\ref{cor:split} for $k_0$ and $\delta_0=64k_0m+2|E(T)|(|E(T_1)|+|E(T_2)|)$. Let $k=16k_0$ and $\delta=d_0+2k$.
%TODO choice of delta and k
%Let $m=\max(|E(T_1)|, |E(T_2)|)$ and let $k_0=K(T_1,T_2)$, where $K(T_1,T_2)$ is as in Lemma~\ref{lem:conn-twotrees}. Let $d_0$ be as in Corollary~\ref{cor:split} for $k_0$ and $\delta_0=64k_0m$. 

%Let $k=16k_0$ and $\delta=d_0+2k$. Since $k\geq K(T_1,T_2)$, if $G$ is $k$-edge-connected, the result immediately follows from Lemma~\ref{lem:conn-twotrees}. 

Let $(H_0, G_0):= (\emptyset, G)$ and we repeat the following recursive procedure, obtaining pairs of stub graphs $(H_i,G_i)$ until $G_i$ is empty. Let $n$ be the number of steps before $G_{n}$ is empty.

If $G_i$ is $k$-edge-connected or has only one vertex, we let $H_{i+1}=G_i$ and $G_{i+1}=\emptyset$.
Assume that $G_i$ is not $k$-edge-connected and has more than one vertex. Then, by Lemma~\ref{lem:cut}, there exists a cut $(A,B)$ in $G_i$ of order at most $2k$ such that $G[A]$ is $k$-edge-connected or has only one vertex.

Let $H_{i+1}=G_i[A]$ and let $C_{i+1}$ be the set of edges $uv\in E(G_i)$ with $u\in A$ and $v\in B$. Let $G_{i+1}$ be the stub graph obtained from $G_i[B]$ by adding a stub with index $i+1$ incident with $v$ for every edge in $C_{i+1}$ incident with $v$. 

Observe that $V(G)=\dot{\bigcup}_{i=0}^{n}V(H_i)$. Moreover, the graphs $G_i$ and $H_i$ have the following properties.

\begin{claim}
For every $i\in [n]$, the following holds:
\begin{itemize}
\item There are at most $2k$ stubs with index $i$ created during the procedure (i.e., in total, there are at most $2k$ stubs with index $i$ in the stub graphs $H_1, \ldots, H_n$),
\item $G_i$ has minimum degree at least $\delta$,
\item $H_i$ is $k$-edge-connected or has only one vertex,
\item $H_i$ has minimum degree at least $\delta-2k$ (minimum of the empty set is $\infty$).
%\item index of each stub $G_i$ is less than $i$ and index of each stub $H_i$ is less than $i-1$, 
\end{itemize} 
\end{claim}

\begin{inproof}
Since there is a one-to-one correspondence between the stubs of index $i$ and the edges in $C_i$ and $|C_i|\leq 2k$, there are at most $2k$ stubs with index $i$.  Moreover, $\deg_{G_{i-1}}(v)=\deg_{G_{i}}(v)$ for every $v\in V(G_{i})$ and therefore $G_{i}$ has minimum degree at least $\delta$ by induction. 

By construction, $H_i$ is $k$-edge-connected or has only one vertex and since $G_{i-1}$ has minimum degree at least $\delta$, $H_i$ has minimum degree at least $\delta-2k$ as required.

\end{inproof}

We call a $\{T,T_1,T_2\}$-decomposition {\em balanced} if the numbers of copies of $T_1$ and $T_2$ differ by at most $|E(T)|$.
Since $G_n$ is empty, it has a balanced $\{T,T_1,T_2\}$-decomposition $\TT_n$ (the trivial one).
Next, we proceed inductively, constructing a balanced $\{T,T_1,T_2\}$-decomposition $\TT_{i-1}$ of $G_{i-1}$ from a balanced decomposition $\TT_i$ of $G_i$ for $i\geq 2$. Moreover, in each step we  increase the number of copies of $T_1$ in the decomposition by at most $|E(T)|$ and keep the number of copies of $T_2$ the same, or the other way round, we  increase the number of copies of $T_2$ in the decomposition by at most $|E(T)|$ and keep the number of copies of $T_1$ the same.

In the last step, we construct a $\{T,T_1,T_2\}$-decomposition of $G$ from a balanced $\{T,T_1,T_2\}$-decomposition of $G_1$ in a similar way, ensuring that the numbers of copies of $T_1$ and $T_2$ in the constructed decomposition are the same.

Roughly speaking, each step of construction has two phases: first, we replace every stub $s$ in $G_i$ of index $i$ by a subtree $\st(s)$ in $G_{i-1}$. Then we decompose the remaining part of $G_{i-1}$ using Lemma~\ref{lem:conn-twotrees}. In the last step, when i=1, we will proceed in a slightly different way to ensure that the resulting decomposition will contain the same number of copies of $T_1$ and $T_2$.

More formally, given a balanced $\{T,T_1,T_2\}$-decomposition $\TT_i$ of $G_{i}$, let $j=1$ if the number of copies of $T_1$ in $\TT'_i$ is smaller than the number of copies of $T_2$ and $j=2$ otherwise. Recall that $V(G_{i-1})=V(H_i)\dot{\cup}V(G_i)$ and $E(G_{i-1})=E(H_i)\dot{\cup}C_{i} \dot{\cup} E(G_i)\setminus S_{i}(G_i)$.

If $H_i$ has more than one vertex, it is $k$-edge-connected and thus, by Corollary~\ref{cor:split}, $H_i$ contains a spanning subgraph $R_i$ with minimum degree at least $4km$ such that $H_i'=H_i\setminus E(R_i)$ is $K(T,T_j)$-edge-connected. If $H_i$ has only one vertex, let $R_i=H_i$ and $H'_i$ is an isolated vertex. 

We replace every embedding of $T$, $T_1$ or $T_2$ in $\TT_i$ which contains a stub $s\in S_{i}$ by an embedding of $T$, $T_1$ or $T_2$ in $G_{i-1}\setminus E(H'_i)$. This yields a partial $\{T,T_1,T_2\}$-decomposition $\TT'$ of $G_{i-1}$. 

Moreover, $\TT'$ is such that $E(H'_i)\subseteq E(G_{i-1})\setminus E(\TT')\subseteq E(H_i)$. Thus, the stub graph $H_i''=(V(H_i),E(G_{i-1})\setminus E(\TT'))$ has a $\{T,T_j\}$-decomposition $\TT''$ that contains at most $|E(T)|$ copies of $T_j$ by Lemma~\ref{lem:conn-twotrees} and by Observation~\ref{obs:stub-dec}, because $H_i''$ is either $K(T,T_j)$-edge-connected or has only one vertex (and therefore no edges). 
Then, $\TT_{i+1}=\TT'\cup\TT''$ forms a $\{T,T_1,T_2\}$-decomposition of $G_{i-1}$. Since there are at most $|E(T)|$ copies of $T_j$ in $\TT'$, from the choice of $j$ it follows that if the difference between the number of copies of $T_1$ and $T_2$ in $\TT_i$ was at most $|E(T)|$, the difference in $\TT_{i-1}$ is also at most $|E(T)|$.
%TODO in the last step, we have to be more careful

Let $K\in\TT_i$ be an embedding of $T$, $T_1$ or $T_2$ containing a stub $s$ with index $i$ (i.e., $s\in S_{i}(G_i)$). Note that $K$ contains at most one stub in $S_{i}(G_i)$. For each such $K$, we construct $K'\in \TT'$ such that $K\setminus S_i(G_i)\subseteq K'$ in the following way. We assume that $K$ is an embedding of $T$, the construction for $T_1$ and  $T_2$ is analogous. 

Let $\II$ be the set of the indices of the stubs in $K$. Let $v$ be the vertex incident with $s$ and let $uv$ be the edge in $C_{i}$ corresponding to the stub $s$. 
We find an embedding $S_s$ of $\st(s)$ such that
\begin{itemize}
\item $v\in V(S_s)$ and it corresponds to the root of $\st(s)$,
\item $uv\in E(S_s)$ and $S_s\setminus v\subseteq R_i$, and
\item no stub in $S_s$ has its index in $\II$.
\end{itemize}

Then, $K':=S_s \cup (K\setminus s)$ is an embedding of $T$ in $G_{i-1}$.

Moreover, we ensure that $E(S_{s_1})$ and $E(S_{s_2})$ are disjoint for every two distinct stubs $s_1, s_2\in S_i(G_i)$. Thus, the edge sets of the embeddings in $\TT'$ will be mutually disjoint. 

We construct an embedding $S_s$ of $\st(s)$ greedily. 
Starting from $S_s$ which consists of the root in $v$ and the edge $uv$, we add edges and stubs one by one. 
At the same time, we remove the used edges and stubs from $R_i$, making sure that no edge and no stub is used in more than one embedding. 
Let $\II'$ be the set of the indices of the stubs in $S_s$. 
Assume that $S_s$ is not yet an embedding of $\st(s)$ and let $w\in R_i$ be a vertex of $S_s$ to which we need to add an edge or a stub. 
We argue that either $w$ has a neighbor $w'$ in $R_i\setminus S_s$ and therefore we can extend $S_s$ by the edge $ww'$ (removing $ww'$ from $R_i$) or $w$ is incident with a stub in $R_i$ such that its index is not in $\II\cup \II'$. 
This is indeed the case; at the beginning, $w$ had degree at least $4km$ in $R_i$, at most $(2k-1)m$ edges and stubs from $R_i$ were removed by embedding trees corresponding to stubs in $S_{i}(G_i)\setminus \{s\}$ and at most $m$ edges and stubs incident with $v$ were removed or cannot be used for extending $S_s$ because the other endpoint of the edge is already in $S_s$. 
This leaves at least $2km$ available edges and stubs incident with $v$. 
Since $|\II\cup \II'|< m$, $R_i$ contains a stub incident with $w$ such that its index is not in $\II\cup \II'$, or an edge $ww'$ with $w'\notin S_s$.

At the last step, it remains to construct a $\{T,T_1,T_2\}$-decomposition of $G$ from a balanced $\{T,T_1,T_2\}$-decomposition of $G_1$. Note that $H_1$ is not a single vertex and has no stubs. Let $R_1$, $H_1'$ be subgraphs of $H_1$ defined as before, in particular, $H'_1$ is $k_T$-edge-connected (since $k_0\geq k_T$).

%and $\TT'$ a partial $\{T,T_1,T_2\}$-decomposition of $G$ obtained from $\TT_1$ by expanding stubs as defined before.

Let $t$, $t_1$ and $t_2$ be the numbers of embeddings of $T$, $T_1$ and $T_2$ in $\TT_1$ respectively. Without loss of generality, we assume that $t_1\leq t_2$.

%Let $z=|E(G)|-t|E(T)|-t_1|E(T_1)|-t_2|E(T_2)|$, i.e, $z$ is the number of edges of $G$ .
%Then, 
Since $|E(G)|$ is divisible by the greatest common divisor of $|E(T)|$ and $|E(T_1)|+|E(T_2)|$, %z-|E(T_1)|(t_2-t_1)= %
$|E(G)|- t|E(T)|-t_2(|E(T_1)|+|E(T_2)|)$ is also divisible by the greatest common divisor of $|E(T)|$ and $|E(T_1)|+|E(T_2)|$. By Observation~\ref{obs:bezout}, there exists an integer $0\leq t'<|E(T)|$ such that $|E(G)|- t|E(T)|-t_2(|E(T_1)|+|E(T_2)|)-t'(|E(T_1)|+|E(T_2)|)$ is divisible by $|E(T)|$.

We greedily construct $t'+(t_2-t_1)$ copies of $T_1$ and $t'$ copies of $T_2$ in $R_1$, denote the resulting partial $\{T_1,T_2\}$-decomposition by $\TT^{*}$ and remove its edges from $R_1$. Note that the minimum degree of $R_1$ decreases by at most $2|E(T)|(|E(T_1)|+|E(T_2)|)$ and thus it is still at least $4km$. Thus, we can proceed in the same way as above, i.e., we construct a partial $\{T,T_1,T_2\}$-decomposition $\TT'$ from $\TT_1$ by expanding the stubs, using the remaining edges of $R_1$.

Then, $\TT^{*}\cup \TT'$ is a partial $\{T,T_1,T_2\}$-decomposition that contains the same number of copies of $T_1$ and $T_2$.
As before, let $H_1''$ be a graph obtained from $H_1'$ by adding unused edges of $R_1$. By our choice of $t'$, the number of edges of $H_1''$ divisible by $|E(T)|$ and thus by Theorem~\ref{thm:barat-thomassen} it has a $T$-decomposition $\TT''$. Then, $\TT^*\cup \TT'\cup \TT''$ is a  $\{T,T_1,T_2\}$-decomposition that contains the same number of copies of $T_1$ and $T_2$.
\end{proof}

%The crucial obstruction to extending this approach to prove the existence of a $T$-decomposition is the necessity for a graph to have its number of edges divisible by $|E(T)|$ in order to have an $T$-decomposition.
%We do not have any control of how many edges are needed for expanding the stubs and thus it is impossible to guarantee that $H_i''$ satisfies any divisibility r%equirements. This problem can be, to some extent, overcome by introducing "length" of a stub $s$, which specifies the number of edges of $\st(s)$. We implement this idea in~\cite{bib:our-paths} to show that for any given path $P$, every $3$-edge connected graph with sufficiently high degree and number of edges divisible by three has a $P$-decomposition.

\section*{Acknowledgements}
Part of this work was done while the first author was a postdoc at Laboratoire d’Informatique du Parall\'elisme,
\'Ecole Normale Sup\'erieure de Lyon,
69364 Lyon Cedex 07, France.

\bibliography{biblio} 
\bibliographystyle{abbrv}

\end{document}